\documentclass[11pt,oneside]{amsart}

\usepackage{bm}
\usepackage{color}
\usepackage{mathrsfs}
\usepackage{geometry, epsf,graphicx,amstext,url}
\usepackage{amsmath,amsthm,amssymb}
\usepackage[breaklinks,colorlinks,linkcolor=black,citecolor=black,urlcolor=black]{hyperref}
\usepackage{enumerate}
\usepackage{float}
\usepackage{tikz}
\newtheorem{defi}{\bf Definition}
\newtheorem{thm}[defi]{\bf Theorem}
\newtheorem{lm}[defi]{\bf Lemma}
\newtheorem{cor}[defi]{\bf Corollary}

\newtheorem{prp}[defi]{\bf Proposition}

\def\meas{\operatorname{meas}}

\begin{document}

\title{On the number variance of sequences with small additive energy}
\date{\today}
\author{Zonglin Li}
\address{School of Mathematics, University of Bristol, Bristol BS8 1UG, U.K.}
\email{zonglin.li@bristol.ac.uk}
\author{Nadav Yesha}
\address{Department of Mathematics, University of Haifa, 3498838 Haifa, Israel.}
\email{nyesha@univ.haifa.ac.il}

\thanks{We thank Jens Marklof and Zeév Rudnick for helpful discussions and comments. Zonglin Li is supported by the China Scholarship Council [202008060338]. Nadav Yesha is supported by the ISRAEL SCIENCE FOUNDATION (grant No.~1881/20).}

\begin{abstract}
	For a real-valued sequence $ (x_n)_{n=1}^\infty $, denote by $ S_N(\ell) $ the number of its first $ N $ fractional parts lying in a random interval of size $ \ell := L/N $, where $ L=o(N) $ as $ N\to \infty $. We study the variance of $ S_N(\ell) $ (the number variance) for sequences of the form $ x_n = \alpha a_n$, where $(a_n)_{n=1}^\infty $ is a sequence of distinct integers. We show that if the additive energy of the sequence $(a_n)_{n=1}^\infty $ is bounded from above by $ N^{5/2-\varepsilon}/L $ for some $ \varepsilon>0 $, then for almost all $ \alpha $, the number variance is asymptotic to $ L $ (Poissonian number variance). This holds in particular for the sequence $ x_n = \alpha n^d, d\ge 2 $ whenever $ L = N^{\beta}$ with $ 0 \le \beta<1/2 $.
\end{abstract}

\maketitle
%\tableofcontents
\section{Introduction}
The study of the distribution of fractional parts of real-valued sequences has a long history which goes back to the pioneering work of Weyl \cite{jr:Weyl}. The basic notion in this theory is of \emph{uniform distribution modulo one}: we say that a real-valued sequence $ \left( x_n \right)_{n=1}^\infty $ is uniformly distributed modulo one (u.d. mod 1) if for any interval $ I \subseteq [0,1) $, we have
\[ \lim\limits_{N\to \infty} \frac{1}{N} \# \left\{ 1\le n \le N : \{x_n\}  \in I \right\} = \text{length}(I) \]
where $ \{ x \} $ denotes the fractional part of $ x $.
A simple example of a u.d. mod 1 sequence is the Kronecker sequence $ x_n = \alpha n $ where $ \alpha \in \mathbb{R} $ is irrational. This was extended by Weyl \cite{jr:Weyl} to  polynomial sequences $ x_n = p(n) $, where $ p(x)=a_d x^d +\dots +a_1 x+a_0 $ is a real polynomial such that at least one of its non-constant coefficients $ a_1 , \dots, a_d $ is irrational. In particular, the monomial sequence  $ x_n = \alpha n^d $ where $\alpha$ is irrational and $d$ is a positive integer is  u.d. mod 1.

If one is interested in detecting pseudo-random behaviour of the fractional parts of a sequence, the mere notion of u.d. mod 1 is far from satisfactory. A better way to capture such behaviour is to consider finer-scale statistics, such as the  distribution of the gaps between neighbouring elements of the first $ N $ fractional parts of the sequence. For a pseudo-random sequence, this distribution will converge (after rescaling by $ 1/N $, the size of the average gap)  to the exponential distribution, which is the gap distribution of a sequence of independent, uniformly distributed random points in the unit interval (Poissonian gap distribution).

Since consecutive elements of the sequence are not necessarily neighbouring modulo one, the gap distribution is rather difficult to  study; nevertheless, one can work with the simpler \emph{pair correlation function} which detects \emph{all} pairs of elements modulo one and is therefore easier to analyse.
Additionally, if we restrict to sequences of the form $ x_n = \alpha a_n $ where $ (a_n)_{n=1}^\infty $ is a real-valued sequence, it is often the case that only little can be said for specific values of $ \alpha $, and therefore most of the available results for such sequences hold in a \emph{metric} sense, i.e., for almost all $ \alpha$.
Rudnick and Sarnak \cite{jr:RudnickSarnak} proved that for almost all $ \alpha$, the pair correlation function of the monomial sequence  $ x_n = \alpha n^d, d\ge 2 $ has a Lebesgue limiting distribution, which is consistent with the random model (Poissonian pair correlation).
Another example is $ x_n = \alpha a_n $ where $ (a_n)_{n=1}^\infty $ is a \emph{lacunary} sequence (i.e., $ a_{n+1} / a_n \ge C > 1 $ for all $ n $); for integer-valued lacunary sequences, Rudnick and Zaharescu \cite{jr:RudnickZaharescu} proved that for almost all $ \alpha $, the pair correlation, as well as all the higher level correlations, are Poissonian (which implies Poissonian gap distribution by a well-known argument, see, e.g., \cite{jr:KurlbergRudnick}); this was recently extended to real-valued lacunary sequences by Chaubey and Yesha \cite{jr:ChaubeyYesha}.
 
Given a sequence $ \mathcal{A} := (a_n)_{n=1}^\infty $ of \emph{distinct integers}, Aistleitner, Larcher and Lewko \cite{jr:AistleitnerLarcherLewko} obtained a streamlined criterion for metric Poissonian pair correlation of the sequence $ x_n = \alpha a_n $ in terms of the additive energy of $ \mathcal{A} $, which is defined by
\[
E_N(\mathcal{A})=\#\{(i,j,k,l)\in[1,N]^4:a_i+a_j=a_k+a_l\};
\]
that is, they showed that if $ E_N(\mathcal{A}) \ll N^{3-\varepsilon} $ for some $ \varepsilon>0 $, then $ (x_n)_{n=1}^\infty  $ has Poissonian pair correlation for almost all $ \alpha $. For example, it follows from the proofs of \cite{jr:RudnickSarnak,jr:RudnickZaharescu} that for $ a_n = n^d, d\ge2$ we have  $ E_N(\mathcal{A}) \ll N^{2+o(1)} $, and for integer-valued lacunary $ (a_n)_{n=1}^\infty $ we have $ E_N(\mathcal{A}) \ll N^2 $, so that metric Poissonian pair correlation of these sequences follows directly from the above criterion.

While the gap distribution and the pair correlation  are small-scale statistics which measure the behaviour of the sequence in the scale of the mean gap $ 1/N $, one can also consider intermediate-scale statistics such as the number variance of the sequence, i.e., the variance of the number of fractional parts of the sequence in random intervals of size $ L/N $, where $ L=L(N)\to \infty $ and $ L=o(N)$. For a random sequence of points, the number variance is asymptotic to $ L $ (Poissonian number variance), and we therefore expect the same behaviour for a pseudo-random sequence. In a recent paper \cite{jr:Yesha}, this was shown to hold in the metric sense for dilations of real-valued lacunary sequences $ (a_n)_{n=1}^\infty $ whenever $ L =  N^\beta$ with $ 0\le\beta<1/2 $; our goal is to formulate a criterion for metric Poissonian number variance for dilations of integer-valued sequences in terms of the additive energy, that will in particular show such behaviour for the monomial sequence  $ x_n = \alpha n^d, d\ge 2 $ (and will give an alternative proof to \cite{jr:Yesha} in the case of integer-valued lacunary $ (a_n)_{n=1}^\infty $).

\subsection{Statement of the main result}

Let $\chi$ denote the characteristic function of the interval $[-\tfrac{1}{2},\tfrac{1}{2})$.
The periodic characteristic function of the interval $[x_0-\tfrac{\ell}{2},x_0+\tfrac{\ell}{2})+\mathbb{Z}$ around $x_0$ of size $0<\ell := L/N \le 1$ can be written as
\[
\chi_\ell(x)=\sum_{n\in\mathbb{Z}}\chi\left(\frac{x-x_0+n}{\ell}\right).
\]
The number of the first $N$ elements of the sequence $(x_n)_{n=1}^\infty$ lying in the interval is
\[
S_N(\ell)=\sum_{j=1}^N\chi_\ell(x_j).
\]
If the centre $x_0$ is chosen uniformly randomly in the unit interval, $S_N(\ell)$ becomes a random variable,
whose expected value is
\[
\langle S_N(\ell)\rangle=\int_0^1S_N(\ell)\,dx_0=\sum_{j=1}^N\sum_{n\in\mathbb{Z}}\int_0^1\chi\left(\frac{x_j-x_0+n}{\ell}\right)\,dx_0=L.
\]
The number variance is defined by
\begin{equation}
\Sigma_N^2(L):=\langle (S_N(\ell)-L)^2\rangle=\langle S_N(\ell)^2\rangle-L^2.
\label{eq:Num Var}\end{equation}
For a sequence of the form $ x_n = \alpha a_n $, we denote the number variance by $ \Sigma_N^2(L,\alpha) $ to indicate the dependence on $ \alpha $. We now state our main result, namely, that the number variance is Poissonian for almost all $ \alpha $ in a suitable regime which depends on the additive energy. To simplify, we restrict to the case $ L=N^\beta$, $0\le \beta<1/2 $, though a similar result for more general functions $ L=L(N) $  may be derived by the same method under a mild condition on the oscillations of $ L $ as in \cite{jr:Yesha}.

\begin{thm}Let $ \mathcal{A}=(a_n)_{n=1}^\infty $ be a sequence of distinct integers, and let $L= N^\beta$ with $ 0\le \beta <1/2 $. Assume that $E_N(\mathcal{A})\ll N^{5/2-\varepsilon}/L$ for some $ \varepsilon>0 $. Then as $N\to\infty$,
\begin{equation}\label{eq:PoissonianNV}
    \Sigma_N^2(L,\alpha)=L+o(L)
\end{equation}
for almost all $\alpha\in [0,1]$.
\label{thm: main}\end{thm}

For sequences with small additive energy, e.g., $ a_n = n^d, d\ge2$ or integer-valued lacunary sequences, we immediately conclude:

\begin{cor}Let $ \mathcal{A}=(a_n)_{n=1}^\infty $ be a sequence of distinct integers, and let $ L= N^\beta$ with $ 0\le \beta <1/2 $. If $ E_N(\mathcal{A}) \ll N^{2+o(1)} $, then \[  \Sigma_N^2(L,\alpha)=L+o(L) \]for almost all $\alpha\in [0,1]$.
\end{cor}

\subsection{The pair correlation function}

For a compactly supported function $f$ and $ L>0 $, the pair correlation function of a sequence $(x_n)_{n=1}^\infty$ is defined by
\begin{equation}
R_N^2(L,f):=\frac{1}{N}\sum_{1\leq i\neq j\leq N}\sum_{m\in\mathbb{Z}}f\left(\frac{x_i-x_j+m}{\ell}\right).
\label{def:Pair Cor}\end{equation}
As usual, for sequences of the form $ x_n = \alpha a_n $, we will stress the $ \alpha $ dependence and denote the pair correlation function by $ R_N^2(L,\alpha ,f) $.
If $ L$ is not fixed but rather $ L=L(N)\to \infty $ as $ N\to \infty $, e.g., $ L = N^\beta $ with $ 0<\beta<1 $, we will call $  R_N^2(L,\alpha ,f) $ the \emph{long-range pair correlation function}. In Section \ref{sec:NumberVariance}, we will see that the condition \eqref{eq:PoissonianNV} is equivalent to 
\begin{equation}
\label{eq:PoissonianPC} R_N^2(L,\alpha,\Delta) = L +o(1)
\end{equation}
where $ \Delta $ is the tent function defined by
\begin{equation}\label{eq:DeltaDef}
\Delta(x):=\int_\mathbb{R}\chi(x+x_0)\chi(x_0)\,dx_0=\int_{-\frac{1}{2}}^\frac{1}{2}\chi(x+x_0)\,dx_0=\max\{1-|x|,0\}.
\end{equation}
We thus see that a weaker notion of Poissonian long-range pair correlation in the sense of $ R_N^2(L,\alpha,f) = L +o(L) $, as established, e.g., in \cite{jr:Hille, jr:Lutsko, jr:TechnauWalker}, would be insufficient for deriving our results on Poissonian number variance.

\begin{figure}[htp!]
	\centering
	\begin{tikzpicture}
	\draw[->](-2,0)--(2,0)node[right]{$x$};
	\draw[->](0,-0.05)--(0,1.7)node[right]{$y$};
	\foreach \x in {-1,0,1}{\draw(\x,-0.1)--(\x,0.1)node[below,outer sep=5pt]at(\x,0){\x};}
	\foreach \y in {1}{\draw(-0.1,\y)--(0.1,\y)node[left,outer sep=5pt]at(0,\y){\y};}
	\draw[color=black, thick,smooth,domain=0:1]plot(\x,1-\x);
	\draw[color=black, thick,smooth,domain=-1.8:-1]plot(\x,0);
	\draw[color=black, thick,smooth,domain=-1:0]plot(\x,1+\x);
	\draw[color=black, thick,smooth,domain=1:1.8]plot(\x,0);
	\put(20,20){$\Delta(x)$}
	\end{tikzpicture}
	\caption{The graph of the tent function $\Delta$.}
\end{figure}
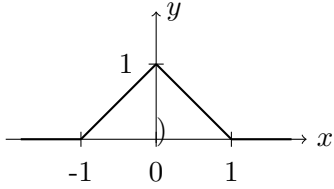

\section{The number variance}
\label{sec:NumberVariance}

In this section, we go along the lines of \cite{jr:Marklof} to recall how the number variance can be expressed in terms of the pair correlation function \eqref{def:Pair Cor}. First, note that
\[
\begin{aligned}
    \langle S_N(\ell)^2\rangle &=\sum_{i,j=1}^N\sum_{m,n\in\mathbb{Z}}\int_0^1\chi\left(\frac{x_i-x_0+m}{\ell}\right)\chi\left(\frac{x_j-x_0+n}{\ell}\right)\,dx_0 \\
    &=\sum_{i,j=1}^N\sum_{m\in\mathbb{Z}}\int_\mathbb{R}\chi\left(\frac{x_i-x_0+m}{\ell}\right)\chi\left(\frac{x_j-x_0}{\ell}\right)\,dx_0.
\end{aligned}
\]

Recalling the definition \eqref{eq:DeltaDef} of $ \Delta  $ and substituting $x'_0=x_j-\ell x_0$ and $dx'_0=-\ell dx_0$, we obtain
\[
\begin{aligned}
    \Delta\left(\frac{x_i-x_j+m}{\ell}\right) &=\int_\mathbb{R}\chi\left(\frac{x_i+\ell x_0-x_j+m}{\ell}\right)\chi(x_0)\,dx_0 \\
    &=\frac{1}{\ell}\int_\mathbb{R}\chi\left(\frac{x_i-x_0'+m}{\ell}\right)\chi\left(\frac{x_j-x_0'}{\ell}\right)\,dx_0'
\end{aligned}    
\]
so that
\begin{equation}
    \langle S_N(\ell)^2\rangle=\ell\sum_{i,j=1}^N\sum_{m\in\mathbb{Z}}\Delta\left(\frac{x_i-x_j+m}{\ell}\right).
\label{eq:E(S_N^2)}\end{equation}

For any non-zero integer $m$ and $0<\ell\leq1$, it is easy to see that $\left|\frac{m}{\ell}\right|\geq 1$ and therefore
\[
    \Delta\left(\frac{m}{\ell}\right)=\max\left\{1-\left|\frac{m}{\ell}\right|,0\right\}=
    \begin{cases}
        1, & m=0 \\
        0, & m\neq 0.
    \end{cases}
\]
Thus, the diagonal terms in equation \eqref{eq:E(S_N^2)} can be easily evaluated and are equal to
\begin{equation}
    \ell\sum_{i=1}^N\sum_{m\in\mathbb{Z}}\Delta\left(\frac{m}{\ell}\right)=L.
\label{eq:Sum Delta}\end{equation}

The Fourier transform of the tent function $\Delta(x)$ is 
\[
    \widehat{\Delta}(x) = \int_\mathbb{R}\Delta(y)e(-xy)\,dy = \int_{-1}^1\left(1-|y|\right)e^{-2\pi ixy}\,dy=
    \begin{cases} 
    1, & x=0 \\
    \frac{\sin^2{\pi x}}{\pi^2x^2}, & x \neq 0
    \end{cases}
\]
where we used the standard notation $e(z):=e^{2\pi iz}$. An application of the Poisson summation formula yields
\[
    R_N^2(L,\alpha,\Delta)=\frac{L}{N^2}\sum_{n\in\mathbb{Z}}\widehat{\Delta}\left(\frac{Ln}{N}\right)\sum_{1\leq i\neq j\leq N}e(n\alpha(a_i-a_j));
\]
combining with \eqref{eq:Num Var}, \eqref{def:Pair Cor} and \eqref{eq:Sum Delta}, we derive the identity
\[
    \Sigma_N^2(L,\alpha)=L-L^2+LR_N^2(L,\alpha,\Delta).
\]

We thus conclude that the asymptotics \eqref{eq:PoissonianNV} and \eqref{eq:PoissonianPC} are equivalent.

\section{Bounding the variance}

In what follows, we denote by $ \mathcal{A} = (a_n)_{n=1}^\infty $ a sequence of distinct integers. As a function of $\alpha\in \mathbb{R}$, the function $R_N^2(L,\alpha,\Delta)$ is periodic and its Fourier expansion can be expressed by
\[
    R_N^2(L,\alpha,\Delta)=\sum_{k\in\mathbb{Z}}b_{k,N}(L)e(k\alpha).
\]
For $k\not=0$, the Fourier coefficients can be written as
\[
\begin{aligned}
    b_{k,N}(L)&=\int_0^1R_N^2(L,\alpha,\Delta)e(-k\alpha)\,d\alpha\\
    &=\frac{L}{N^2}\sum_{n\in\mathbb{Z}}\widehat{\Delta}\left(\frac{Ln}{N}\right)\sum_{1\leq i\not=j\leq N}\int_0^1e((na_i-na_j-k)\alpha)\,d\alpha \\
    &=\frac{L}{N^2}\sum_{n\not=0}\sum_{\substack{1\leq i\not=j\leq N\\n(a_i-a_j)=k}}\widehat{\Delta}\left(\frac{Ln}{N}\right),
\end{aligned}
\]
where the last identity follows since the equation $n(a_i-a_j)=k$ with $n=0$ implies $k=0$. The mean of $R_N^2(L,\alpha,\Delta)$ over $\alpha\in[0,1]$ is 
\begin{equation*}
\begin{aligned}
    \langle R_N^2(L,\alpha,\Delta)\rangle&=b_{0,N}(L)=\frac{L}{N^2}\sum_{n\in\mathbb{Z}}\widehat{\Delta}\left(\frac{Ln}{N}\right)\sum_{\substack{1\leq i\not=j\leq N\\n(a_i-a_j)=0}}1 \\
    &=\frac{L}{N^2}\sum_{1\leq i\neq j\leq N}\widehat{\Delta}(0)=L-\frac{L}{N},
\end{aligned}    
\end{equation*}
where the next to last identity holds since the sequence $ \mathcal{A} $ consists of \emph{distinct} integers. To estimate the variance of $R_N^2(L,\alpha,\Delta)$ as a function of $\alpha$, we define 
\[
    X_N(L,\alpha):=R_N^2(L,\alpha,\Delta)-\langle R_N^2(L,\alpha,\Delta)\rangle =\sum_{k\ne 0}b_{k,N}(L)e(k\alpha).   
\]

We will need the following lemma (see \cite{jr:Rudnick} for an analogous estimate with a smooth test function):

\begin{lm}
For any $a\neq0$, we have
\[
\sum_{n\ne0}\widehat{\Delta}(an)^{2}<\frac{1}{|a|}.
\]
\label{lm: 1/a refinement}\end{lm}
\begin{proof}
Since $\widehat{\Delta}$ is even, we may assume without loss of generality that $a>0$. As $0\le \widehat{\Delta}\leq 1$, we get
\[
\sum_{n\ne0}\widehat{\Delta}(an)^{2}\leq \sum_{n\ne0}\widehat{\Delta}(an).
\]
By the Poisson summation formula, we know that
\[
\sum_{n\in\mathbb{Z}}\widehat{\Delta}(an)=\frac{1}{a}\sum_{m\in\mathbb{Z}}\Delta\left(\frac{m}{a}\right).
\]
For $0<a\leq1$, we have 
\[
\Delta\left(\frac{m}{a}\right)=\begin{cases}
    1, & m=0 \\
    0, & m \neq 0.
\end{cases}
\]
Since $\widehat{\Delta}(0)=1$, we conclude that
\[
\sum_{n\ne0}\widehat{\Delta}(an)^{2}\le\sum_{n \in \mathbb{Z}}\widehat{\Delta}(an) - 1 =\frac{1}{a}\sum_{m\in\mathbb{Z}}\Delta\left(\frac{m}{a}\right)-1=\frac{1}{a}-1<\frac{1}{a}.
\]
For $a>1$, we have
\[
\Delta\left(\frac{m}{a}\right)=\begin{cases}
    1-\tfrac{|m|}{a}, & |m|\leq  a \\
    0, & |m|>a.
\end{cases}
\]
Hence,
\[
\begin{aligned}
    \sum_{n\ne0}\widehat{\Delta}(an)^{2}&\leq\frac{1}{a}\sum_{|m|\leq a}\left(1-\frac{|m|}{a}\right)-1=\frac{1}{a}\left(2\lfloor a\rfloor+1-\frac{\lfloor a\rfloor(\lfloor a\rfloor+1)}{a}\right)-1 \\
    &=\frac{1}{a} \left( \left(1+\lfloor a\rfloor-a\right)+\frac{\lfloor a\rfloor}{a}\left(a-\lfloor a\rfloor-1\right) \right) <\frac{1}{a},
\end{aligned}
\]
where the last inequality holds since $\lfloor a\rfloor \leq a < \lfloor a\rfloor+1$.
\end{proof}

In what follows, we denote
\[
\delta(n) = \begin{cases}
    1, & n=0\\
    0, & n\neq0.
\end{cases}
\]

\begin{lm}
Given non-zero integers $w_r$ and $w_s$, we have
\[
    \frac{L}{N}\sum_{n_1,n_2\neq0}\widehat{\Delta}\left(\frac{Ln_1}{N}\right)\widehat{\Delta}\left(\frac{Ln_2}{N}\right)\delta(n_1w_r-n_2w_s)<\frac{\gcd(w_r,w_s)}{\sqrt{|w_rw_s|}}.
\]    
\label{lm: gcd}\end{lm}
\begin{proof}
Denote the greatest common divisor of $w_r$ and $w_s$ by $d$. The equation $n_1w_r=n_2w_s$ implies that $n_1=n_0w_s/d$ and $n_2=n_0w_r/d$ for some integer $n_0\neq0$. Re-writing the double summation as a single summation, we get
\[
    \frac{L}{N}\sum_{n_1,n_2\neq0}\widehat{\Delta}\left(\frac{Ln_1}{N}\right)\widehat{\Delta}\left(\frac{Ln_2}{N}\right)\delta(n_1w_r-n_2w_s)=\frac{L}{N}\sum_{n_0\neq0}\widehat{\Delta}\left(\frac{Lw_rn_0}{Nd}\right)\widehat{\Delta}\left(\frac{Lw_sn_0}{Nd}\right).
\]
Using the Cauchy-Schwarz inequality, 
\[
    \sum_{n_0\neq0}\widehat{\Delta}\left(\frac{Ln_0w_s}{Nd}\right)\widehat{\Delta}\left(\frac{Ln_0w_r}{Nd}\right)\leq\left(\sum_{n_0\neq0}\widehat{\Delta}\left(\frac{Lw_rn_0}{Nd}\right)^2\right)^{\frac{1}{2}}\left(\sum_{n_0\neq0}\widehat{\Delta}\left(\frac{Lw_sn_0}{Nd}\right)^2\right)^{\frac{1}{2}}.
\]
By Lemma \ref{lm: 1/a refinement}, we conclude that
\[
    \sum_{n_0\neq0}\widehat{\Delta}\left(\frac{Ln_0w_s}{Nd}\right)\widehat{\Delta}\left(\frac{Ln_0w_r}{Nd}\right)<\frac{Nd}{L\sqrt{|w_rw_s|}},
\]
which gives the desired bound.
\end{proof}

We are now ready to estimate the second moment of $ X_N(L,\alpha)$:

\begin{prp}
The second moment of $X_N(L,\alpha)$ is
\[
    \langle X_N(L,\alpha)^2\rangle=\int_0^1\big|X_N(L,\alpha)\big|^2\,d\alpha\ll_\varepsilon LN^{-3+\varepsilon}E_N(\mathcal{A})
\]
for any $\varepsilon>0$.
\label{prp: X_N,L bound with E_N(A)}\end{prp}
\begin{proof}
Set \[W_N(w,\mathcal{A})=\#\{1\leq i\neq j\leq N: a_i-a_j=w\}\] and \[\widetilde{W}_N(n_1,n_2,\mathcal{A})=\#\{(i,j,k,l)\in[1,N]^4:i\neq j, k\neq l, n_1(a_i-a_j)=n_2(a_k-a_l)\}.\]Note that the differences $a_i-a_j$ for $1\leq i \neq j \leq N$ can have at most $N^2-N$ distinct values. For simplicity, let $w_r$ be $N^2$ distinct integers including all possible differences $a_i-a_j$. As for those redundant values $w_r\neq a_i-a_j$, we have $W_N(w_r,\mathcal{A})=0$. We may therefore re-write
\[      
    \widetilde{W}_N(n_1,n_2,\mathcal{A})=\sum_{\substack{1\leq r\leq N^2 \\ 1\leq s\leq N^2}}W_N(w_r,\mathcal{A})W_N(w_s,\mathcal{A})\delta(n_1w_r-n_2w_s).
\]
Consequently, the second moment of $X_N(L,\alpha)$ is
\[
\begin{aligned}
    \int_0^1\big|X_N(L,\alpha)\big|^2\,d\alpha&=\frac{L^2}{N^4}\sum_{n_1,n_2\neq0}\widehat{\Delta}\left(\frac{Ln_1}{N}\right)\widehat{\Delta}\left(\frac{Ln_2}{N}\right)\sum_{\substack{1\leq i\neq j\leq N \\ 1\leq k\neq l\leq N}}\!\delta(n_1(a_i-a_j)-n_2(a_k-a_l)) \\
    &=\frac{L^2}{N^4}\sum_{n_1,n_2\neq0}\widehat{\Delta}\left(\frac{Ln_1}{N}\right)\widehat{\Delta}\left(\frac{Ln_2}{N}\right)\widetilde{W}_N(n_1,n_2,\mathcal{A}).
\end{aligned}    
\]
By Lemma \ref{lm: gcd} and the bound for $ \gcd $ sums \cite{jr:DyerHarman}, we obtain
\[
\begin{aligned}
    \int_0^1\big|X_N(L,\alpha)\big|^2\,d\alpha&\ll\frac{L}{N^3}\sum_{\substack{1\leq r\leq N^2 \\ 1\leq s\leq N^2}}W_N(w_r,\mathcal{A})W_N(w_s,\mathcal{A})\frac{\gcd(w_r,w_s)}{\sqrt{|w_rw_s|}} \\
    &\ll \frac{L}{N^3}\sum_{1\leq r\leq N^2}W_N(w_r,\mathcal{A})^2\exp\left(\frac{10\log r}{\log\log (r+1)}\right).
\end{aligned}
\]
Since
\[
\sum_{1\leq r\leq N^2}W_N(w_r,\mathcal{A})^2=\#\{(i,j,k,l)\in[1,N]^4:i\neq j, k\neq l, a_i-a_j=a_k-a_l\}\leq E_N(\mathcal{A}),
\]
and for $1\leq r \leq N^2$ and $\varepsilon>0$ we have
\[
\exp\left(\frac{10\log r}{\log\log (r+1)}\right)\ll_\varepsilon N^\varepsilon,
\]
we conclude that
\[
    \int_0^1\big|X_N(L,\alpha)\big|^2\,d\alpha\ll_\varepsilon LN^{-3+\varepsilon}E_N(\mathcal{A}).
\]

\end{proof}

\section{Proof of Theorem \ref{thm: main}}

We now use the bound on the second moment of $ X_N(L,\alpha) $ to prove almost sure convergence of $ R_N^2(L,\alpha,\Delta) $ as $ N\to\infty $, first along a subsequence, and then along the full sequence. To this end, we follow the arguments of \cite{jr:Yesha}. 

\begin{lm}
Let $L=L(N)$ such that $ 0<L=o(N) $ as $ N\to\infty $.  For any $\delta>0$ and for any $ \varepsilon>0 $, we have
\[
    \meas\{\alpha\in [0,1]: |\Sigma_N^2(L,\alpha)-L|>\delta L\}=O_{\delta,\varepsilon}( LN^{-3+\varepsilon}E_N(\mathcal{A}))\]
    as $N\to\infty$.
    \label{lm:weak_bound}
\end{lm}
\proof Since
\[
    \frac{\Sigma_N^2(L,\alpha)-L}{L}=X_N(L,\alpha)-\frac{L}{N},
\]
we have
\[
\begin{aligned}
    \meas\{\alpha\in [0,1]: |\Sigma_N^2(L,\alpha)-L|>\delta L \}&=\meas\{\alpha\in [0,1]: |X_N(L,\alpha)-L/N|>\delta\} \\
    &\leq \meas\{\alpha\in [0,1]: |X_N(L,\alpha)|>\delta/2\}
\end{aligned}
\]
for sufficiently large $N$ such that $L/N<\delta/2$. Applying Chebyshev's inequality and Proposition \ref{prp: X_N,L bound with E_N(A)}, we deduce that
\[
\begin{aligned}
    \meas\{\alpha\in [0,1]: &|X_N(L,\alpha)|>\delta/2\}\leq 4\delta^{-2}\int_0^1 |X_N(L,\alpha)|^2\, d\alpha \\
    & \ll_{\delta,\varepsilon} LN^{-3+\varepsilon}E_N(\mathcal{A}),
\end{aligned}
\]
which gives the desired estimate.

\qed

\begin{lm}
Let $L=L(N)$ such that $ 0<L=o(N) $ as $ N\to\infty $. Assume that $E_N(\mathcal{A})\ll N^{5/2-\varepsilon}/L$ for some $ \varepsilon>0 $. Let $N_m=m^2$ and let $L_m=L(N_m)$. Then as $m\to\infty$, we have
\[
    R_{N_m}^2(L_m,\alpha,\Delta)=L_m+o(1)
\]
for almost all $\alpha\in [0,1]$.
\label{lm: L_m+o(1)}\end{lm}
\proof Given any $\delta>0$ and sufficiently large $N$, Lemma \ref{lm:weak_bound} gives
\[
\begin{aligned}
    \meas\{\alpha\in [0,1]:|R_N^2(L,\alpha,\Delta)-L|>\delta\} & = \meas\{\alpha\in [0,1]: |\Sigma_N^2(L,\alpha)-L|>\delta L\} \\ &\ll_{\delta,\varepsilon} LN^{-3+\varepsilon/2}E_N(\mathcal{A})\ll_{\delta,\varepsilon} N^{-1/2-\varepsilon/2}.
\end{aligned}
\]
Thus, we obtain
\[
    \meas\{\alpha\in [0,1]:|R_{N_m}^2(L_m,\alpha,\Delta)-L_m|>\delta\}\ll_{\delta,\varepsilon}m^{-1-\varepsilon},
\]
and therefore
\[
    \sum_{m=1}^\infty\meas\{\alpha\in [0,1]:|R_{N_m}^2(L_m,\alpha,\Delta)-L_m|>\delta\}<\infty.
\]
By the Borel-Cantelli lemma, we conclude that $R_{N_m}^2(L_m,\alpha,\Delta)=L_m+o(1)$ for almost all $\alpha\in [0,1]$.

\qed

We are now ready to prove our main theorem.

\proof[Proof of Theorem \ref{thm: main}] Set $N_m=m^2$, so that $ L_m = m^{2\beta} $. For each $N$, there exists a unique $m$ such that $N_{m-1}\leq N <N_m$. Clearly,
\[
   \frac{N_m}{N}=1+O(1/m),
\] 
so that there exists $C>0$ such that $N_m/N\le 1+C/m$ for sufficiently large $N$. Also, note that
$ L_{m-1} \le L \le L_m $ and that
$ L_m - L_{m-1} = o(1), $ so that $ L_m = L + o(1). $
Thus,
\[
\begin{aligned}
    R_N^2(L,\alpha,\Delta)&\leq \frac{N_m}{N}R_{N_m}^2(LN_m/N,\alpha,\Delta) \\
    &\leq(1+C/m)R_{N_m}^2((1+C/m)L_m,\alpha,\Delta).
\end{aligned}
\]
Since $ 0<\beta<1/2 $, we have \[(1+C/m)L_m=L_m+o(1)=L+o(1).\] Thus, applying Lemma \ref{lm: L_m+o(1)} with $ L(1+C/N^{1/2}) \sim L $, we see that for almost all $ \alpha $, we have
\[
    R_{N_m}^2((1+C/m)L_m,\alpha,\Delta)=(1+C/m)L_m+o(1)=L+o(1).
\]
We conclude that for almost all $ \alpha $,
\[
    R_N^2(L,\alpha,\Delta)\leq(1+C/m)(L+o(1))=L+o(1).
\]
Similarly, for almost all $ \alpha $,
\[
    R_N^2(L,\alpha,\Delta)\geq \frac{N_{m-1}}{N}R_{N_{m-1}}^2(L N_{m-1}/N,\alpha,\Delta)=L-o(1).
\]
Hence, we conclude $R_N^2(L,\alpha,\Delta)=L+o(1)$ and thus $\Sigma_N^2(L,\alpha)=L+o(L)$ for almost all $ \alpha $.

\qed 

\bibliographystyle{plain}
\bibliography{refs}

\end{document}